\documentclass[preprint]{elsarticle}

\makeatletter
\def\ps@pprintTitle{%
	\let\@oddhead\@empty
	\let\@evenhead\@empty
	\def\@oddfoot{\footnotesize\itshape
		{} \hfill\today}%
	\let\@evenfoot\@oddfoot
}
\makeatother

\usepackage{latexsym}
\usepackage{lineno}
\usepackage{hyperref}
\usepackage{indentfirst}
\usepackage{amsxtra}
\usepackage{amssymb}
\usepackage{amsthm}
\usepackage{natbib}
\usepackage{amsmath}
\usepackage{tikz}
\usepackage{amscd}
\usepackage[capitalise]{cleveref}
\newtheorem{theor}{Theorem}
\newtheorem{prop}[theor]{Proposition}
\newtheorem{cor}[theor]{Corollary}
\newtheorem{lemma}[theor]{Lemma}

\theoremstyle{definition} 
\newtheorem{defin}{Definition}

\newtheorem{ex}[theor]{Example}
\newtheorem{exs}[theor]{Examples}

\DeclareMathOperator{\Sym}{Sym}
\DeclareMathOperator{\id}{id}
\DeclareMathOperator{\Aut}{Aut}
\DeclareMathOperator{\Ret}{Ret}

\begin{document}

\begin{frontmatter}
	\title{Left non-degenerate set-theoretic solutions of the Yang-Baxter equation and dynamical extensions of q-cycle sets \tnoteref{mytitlenote}}
	\tnotetext[mytitlenote]{This work was partially supported by the Dipartimento di Matematica e Fisica ``Ennio De Giorgi" - Università del Salento. The authors are members of GNSAGA (INdAM).}
	\author[unile]{Marco~CASTELLI}
	\ead{marco.castelli@unisalento.it}
	\author[unile]{Francesco~CATINO\corref{c1}}
	\ead{francesco.catino@unisalento.it}
	\author [unile] {Paola~STEFANELLI}
	\ead{paola.stefanelli@unisalento.it}
		\cortext[c1]{Corresponding author}
	\address[unile]{Dipartimento di Matematica e Fisica ``Ennio De Giorgi"
		\\
		Universit\`{a} del Salento\\
		Via Provinciale Lecce-Arnesano \\
		73100 Lecce (Italy)\\}

\begin{abstract}
A first aim of this paper is to give sufficient conditions on left non-degenerate bijective set-theoretic solutions of the Yang-Baxter equation so that they are non-degenerate. 
In particular, we extend the results on involutive solutions obtained by Rump in \cite{rump2005decomposition} and answer in a positive way to a question posed by 
Ced\'o, Jespers, and Verwimp
\cite[Question 4.2]{cedo2019structure}. 
Moreover, we develop a theory of extensions for left non-degenerate set-theoretic solutions of the Yang-Baxter equation that allows one to construct new families of set-theoretic solutions. 
\end{abstract}
\begin{keyword}
\texttt{q-cycle set\sep cycle set\sep set-theoretic solution\sep Yang-Baxter equation\sep brace\sep skew brace}
\MSC[2020] 16T25\sep 20N02\sep 20E22\sep 81R50
\end{keyword}

\end{frontmatter}

\section{Introduction}
A \emph{set-theoretic solution of the Yang-Baxter equation}, or shortly a \textit{solution}, is a pair $(X,r)$ where $X$ is a non-empty set and $r$ is a map from $X\times X$ into itself such that 
$$
r_1r_2r_1 = r_2r_1r_2 ,
$$
where $r_1:= r\times id_X$ and $r_2:= id_X\times r$.  Let $\lambda_x:X\to X$ and $\rho_y:X\to X$ be maps such that  
$$
r(x,y) = (\lambda_x(y), \rho_y(x))
$$ 
for all $x,y\in X$. A set-theoretic solution of the Yang-Baxter equation $(X, r)$ is said to be a left [\,right\,] non-degenerate if $\lambda_x\in\Sym(X)$ [\,$\rho_x\in\Sym(X)$\,], for every $x\in X$, and non-degenerate if it is left and right non-degenerate.\\ 
Drinfeld's paper \cite{drinfeld1992some} stimulated much interest in this subject. In recent years, after the seminal papers by Gateva-Ivanova and Van den Bergh \cite{gateva1998semigroups} and Etingof, Schedler and Soloviev \cite{etingof1998set} the involutive solutions $r$, i.e., $r^2=id_{X\times X}$, have received a lot of attention.

To study involutive solutions, many theory involving a lot of algebraic structures has developed. Several examples of involutive solutions have obtained by using groups, racks, and quandles \cite{Ga18,gateva2008matched, cedo2010involutive,Ba18, CeJeOk19x}. 
In 2005, Rump introduced \textit{cycle sets}, non-associative algebraic structures that allow one to find involutive left non-degenerate solutions. 
A set $X$ endowed of an operation $\cdot$ is said to be a cycle set if the left multiplication  $\sigma_x:X\longrightarrow X, y\mapsto x\cdotp y$ is invertible, for every $x\in X$, and the relation 
\begin{align*}
(x\cdotp y)\cdotp (x\cdotp z) 
=
(y\cdotp x)\cdotp (y\cdotp z)
\end{align*}
is satisfied, for all $x,y,z\in X$.
By cycle sets several families of involutive solutions were determined and several interesting results were obtained (see, for example, \cite{dehornoy2015set,cacsp2018quasi,cacsp2019,capiru2020,vendramin2016extensions,cacsp2018,bon2019}). 
%
%

In 2000, Lu, Yan, and Zhu \cite{LuYZ00} and Soloviev \cite{So00} started the study
of non-degenerate solutions that are not necessarily involutive.
To obtain new families of bijective non-degenerate solutions, in 2015 Guarnieri and Vendramin \cite{guarnieri2017skew} introduced the algebraic structure of \textit{skew brace}, a generalisation of the braces introduced by Rump in \cite{rump2007braces}. 
Some works where this structure is studied are
\cite{cacs1,cacs4,cedo2014braces,smock2018skew,CeSmVe19,JeKuVaVe19}, just to name a few.
As skew braces are the analogue version of braces for non-involutive non-degenerate solutions, \textit{q-cycle sets}, introduced recently by Rump \cite{rump2019covering}, are the analogue version of cycle sets for left non-degenerate solutions that are not necessarily involutive. Recall that a set $X$ with two binary operations $\cdotp$ and $:$ is a q-cycle set if $\sigma_x:X\longrightarrow X, y\mapsto x\cdotp y$ is invertible, for every $x\in X$, and 
\begin{align*}
(x\cdotp y)\cdotp (x\cdotp z) &=(y:x)\cdotp (y\cdotp z)\\
(x:y):(x:z) &= (y\cdotp x):(y:z)\\
(x\cdotp y):(x\cdotp z) &= (y:x)\cdotp (y: z)
\end{align*}
hold, for all $x,y,z\in X$. If $(X,\cdotp,:)$ is a q-cycle set, then the pair $(X,r)$, where $r(x,y):=(\sigma_x^{-1}(y),\sigma_x^{-1}(y):x)$, for all $x,y\in X$, is a left non-degenerate solution. Conversely, if $(X,r)$ is a left non-degenerate solution, then the operations $\cdotp$ and $:$ given by $x\cdotp y:=\lambda_x^{-1}(y)$ and $x:y:=\rho_{\lambda_y^{-1}(x)}(y)$, for all $x,y\in X$, give rise to a q-cycle set. Thanks to this correspondence, one can move from left non-degenerate solutions to q-cycle sets.

In the first part of the paper, we focus on non-degeneracy of bijective solutions. In particular, using q-cycle sets, we show that any finite left non-degenerate bijective solutions is right non-degenerate, giving a positive answer to \cite[Question 4.2]{cedo2019structure}. 
In this way, we also extend the corresponding result for finite involutive solutions provided by Rump in \cite[Theorem 2]{rump2005decomposition} in terms of cycle sets, and by Jespers and Okni\'{n}ski in \cite[Corollary 2.3]{JeOk05} 
in terms of monoids of $I$-type. In addition, using the result on finite left non-degenerate solutions, we give sufficient conditions to find a class non-degenerate solutions that include properly the finite left non-degenerate ones.\\
In Section $4$, we introduce an equivalence relation for a q-cycle set which we call \emph{retraction}, in analogy to the retraction of cycle sets, that is compatible with respect to the two operations. Since the retraction of a non-degenerate q-cycle set is again a non-degenerate q-cycle set, 
we obtain an alternative proof of \cite[Theorem 3.3]{JeP18} for non-degenerate solutions $(X,r)$ with the additional property of $r$ bijective. Clearly, we include the result showed in \cite[Lemma 8.4]{LeVe19}.  

The final goal of this paper is to develop a theory of extensions of q-cycle sets. 
Following the ideas of Vendramin for cycle sets \cite{vendramin2016extensions} and of Nelson and Watterberg for biracks \cite{nelson2013birack}, we introduce a suitable notion of \textit{dynamical extension} of q-cycle sets. Even if dynamical extensions are often hard to find, we introduce several examples of dynamical extensions that are relatively easy to compute. 
Moreover, we introduce several families of dynamical extensions that provide non-degenerate solutions that are different from those obtained by skew braces.\\
As an application, we construct a \emph{semidirect product} of q-cycle sets, which is a generalization of the semidirect product of cycle sets introduced by Rump \cite{rump2008semidirect} and, referring to \cite[Problem 4.15]{bardakov2019general}, gives rise to a general definition of semidirect product of biquandles. 

\section{Basic results and examples}
To study non-degenerate solutions, Rump recently introduced in \cite{rump2019covering} the notion of q-cycle set. Recall that a set $X$ together with two binary operations $\cdotp$ and $:$ is a \textit{q-cycle set} if the function $\sigma_x:X\longrightarrow X,$ $y\mapsto x\cdotp y$ is bijective, for every $x\in X$, and
\begin{align}
(x\cdotp y)\cdotp (x\cdotp z) &= (y:x)\cdotp (y\cdotp z)\label{ug1}\\
(x:y):(x:z) &= (y\cdotp x):(y:z)\label{ug2}\\
(x\cdotp y):(x\cdotp z) &= (y:x)\cdotp (y: z)\label{ug3}
\end{align}
hold for all $x,y,z\in X$. 
Hereinafter, we denote by $\mathfrak{q}$ and $\mathfrak{q'}$ the squaring maps related to $\cdotp$ and $:$, i.e., the maps given by
$$
\mathfrak{q}(x):=x\cdotp x 
\qquad\text{and}\qquad
\mathfrak{q'}(x):=x: x,
$$ 
for every $x\in X$. \\
A q-cycle set $(X,\cdotp,:)$ is said to be \textit{regular} if the function $\delta_x:X\longrightarrow X,$ $y\mapsto x:y$ is bijective, for every $x\in X$, and \textit{non-degenerate} if it is regular and $\mathfrak{q}$ and $\mathfrak{q'}$ are bijective.  
At first sight, the notion of non-degeneracy introduced by Rump seems different, but  \cite[Corollary 2]{rump2019covering} ensures that the two definitions are equivalent.\\
The left non-degenerate solution $(X,r)$ provided by a q-cycle set $X$ is given by 
\begin{align*}
r(x,y):= (\sigma_{x}^{-1}(y),\sigma_{x}^{-1}(y):x),
\end{align*}
for all $x,y\in X$, is a left-non degenerate solution. Conversely, if $(X,r)$ is a left non-degenerate solution and $r(x,y)=(\lambda_x(y),\rho_y(x))$, for all $x,y\in X$, then the operations $\cdotp$ and $:$ given by
$$x\cdotp y:=\lambda_x^{-1}(y) \qquad  x:y:=\rho_{\lambda_y^{-1}(x)}(y)$$
for all $x,y\in X$, give rise to a q-cycle set, see \cite[Proposition 1]{rump2019covering}. As one would expect, non-degenerate q-cycle sets corresponds to non-degenerate bijective solutions. 
\medskip

Q-cycle sets allow us to construct a lot of families of solutions obtained in several recent papers. To show this, in the rest of this section we collect several examples of q-cycle sets and highlight some connections with other algebraic structures.\\
At first, note that q-cycle sets with $x\cdotp y=x:y$, for all $x,y\in X$, correspond to cycle sets.
Clearly, in this case, \eqref{ug1}, \eqref{ug2}, and \eqref{ug3} are reduced to 
\begin{align}\label{eq:cycle-set}
(x\cdotp y)\cdotp (x\cdotp z) 
=
(y\cdotp x)\cdotp (y\cdotp z).
\end{align}
These examples of q-cycle sets determine the ``celebrated'' left non-degenerate solutions that are also involutive, that in particular are maps $r:X\times X\to X\times X$ of the form
    \begin{align*}
        r\left(x, y\right) 
        = \left(\sigma^{-1}_x\left(y\right),\sigma^{-1}_x\left(y\right)\cdotp x\right),
    \end{align*}
    for all $x,y\in X$.
    
In \cite{rump2019covering} Rump showed some connections between q-cycle sets and other algebraic structures. In particular, he observed that skew braces correspond to particular q-cycle sets for which the squaring maps $\mathfrak{q}$ and $\mathfrak{q'}$ coincide (see \cite[Corollary 2]{rump2019covering}). 
In this context, we note that an analogue result follows for left semi-braces, algebraic structures
that are a generalisation of skew braces.
We specify that, in this paper, for a left semi-brace we mean the structure introduced in \cite{cacs3}, named left cancellative left semi-brace in \cite{JeVa19}.
Specifically, we say that a set $B$ with two operations $+$ and $\circ$ is  a \emph{left semi-brace} if $\left(B,+\right)$ is a left cancellative semigroup, $\left(B, \circ\right)$ is a group, and 
\begin{align*}
    a\circ\left(b + c\right) = a\circ b + a\circ\left(a^- + c\right) 
\end{align*}
holds, for all $a,b,c\in B$, where $a^-$ is the inverse of $a$ with respect to $\circ$.\\
If $\left(B, +, \circ\right)$ is a left semi-brace, by \cite[Theorem 9]{cacs3} we have that the map $r_B:B\times B\to B\times B$ defined by
\begin{align*}
    r_B\left(a, b\right) 
    = 
    \left(a\circ\left(a^- + b\right), \left(a^- + b\right)^-\circ b\right),
\end{align*}
for all $a,b\in B$, is a left non-degenerate solution.
Hence, the q-cycle set associated to $B$ is the structure $\left(B, \cdotp, :\right)$ where $\cdotp$ and $:$ are given by
\begin{align*}
    a\cdotp b 
    &= \lambda_{a^-}\left(b\right) 
    = a^-\circ\left(a + b\right)\\
    a : b 
    &= \left(\rho_a\left(b^-\right)\right)^-
    = a^-\circ\left(b + a\right) ,
\end{align*}
for all $a,b\in B$. Therefore, it is now easy to see that the squaring maps $\mathfrak{q}$ and $\mathfrak{q'}$ coincide also for q-cycle sets obtained by left semi-braces that are not skew-braces. 
In the following example, we show a family of concrete examples of q-cycle sets determined by left semi-braces.
\begin{ex}\label{ex:semi-brace}
Let $\left(B,\circ\right)$ be a group, $f$ an idempotent endomorphism of $\left(B,\circ\right)$, and $\left(B,+,\circ\right)$ the left semi-brace in \cite[Example 10]{cacs3} where the sum is given by
\begin{align*}
    a + b:= b\circ f\left(a\right),
\end{align*}
for all $a,b\in B$. 
Since the left non-degenerate solution associated to $\left(B, +,\circ\right)$ is the map $r_B:B:B\times B\to B\times B$ given by
\begin{align*}
    r\left(a,b\right) 
    = \left(a\circ b\circ f\left(a\right)^{-1}, f\left(a\right)\right),
\end{align*}
we obtain that $\left(B, \cdotp, :\right)$ is a q-cycle set with $\cdotp$ and $:$ defined by
\begin{align*}
    a\cdotp b:= a^{-1}\circ b\circ f\left(a\right)
    \qquad \qquad
    a : b:= f\left(b\right),
\end{align*}
for all $a,b\in B$. Clearly, if $f\neq \id_B$, then $\left(B, \cdotp, :\right)$ is not regular. 

More in general, let us observe that the hypothesis of idempotency on $f$ is unnecessary to have a structure of q-cycle set on the group $\left(B,\circ\right)$. Obviously, $\left(B, \cdotp,:\right)$ is regular if and only if the map $f$ is bijective.
%
%
\end{ex}
\medskip
%

However, in general, the squaring maps $\mathfrak{q}$ and $\mathfrak{q'}$ do not coincide, as one can see in \cref{ex-reg-q-cycleset} and \cref{ex:left_quasi-normal} below.

\begin{exs}\label{ex-reg-q-cycleset}\hspace{1mm}
\begin{itemize}
    \item[1)] 
    Let $k\in\mathbb{N}$, $B:=\mathbb{Z}/m\mathbb{Z}$, $\cdotp$ and $:$ the binary operations on $B$ given by $x\cdotp y:= y + k$ and $x:y = y$, for all $x,y\in B$. Then, $(B,\cdotp,:)$ is a regular q-cycle set and the solution associated to $(B,\cdotp,:)$ is the map $r:B\times B\to B\times B$ given by
    \begin{align*}
        r\left(x, y\right)
        = \left(y-k, x\right),
    \end{align*}
    for all $x,y\in B$.
    \item[2)] Let $B:=\mathbb{Z}/3\mathbb{Z}$ and $\alpha$ the element of $\Sym(B)$ given by $\alpha:=(0\quad 1)$. Moreover, let $\cdotp$ and $:$ the binary operations on $B$ given by $x\cdotp y:=\alpha(y)$ if $x\in \{0,1\}$ and $x\cdotp y:=y$ otherwise and $x:y:=y$ if $x\in \{0,1\}$ and $x: y:=\alpha(y)$ otherwise. Then, $(B,\cdotp,:)$ is a regular q-cycle set and the solution $r$ associated to $(B,\cdotp,:)$ is defined by
    \begin{align*}
        r\left(x, y\right)
        = \begin{cases}
        \left(\alpha\left(y\right), \ x\right) &\mbox{\quad if \ $x\in\lbrace 0,1 \rbrace$}\\
        \left(y, \ 2\right) &\mbox{\quad if \ $x = 2$}
        \end{cases},
    \end{align*}
    for all $x,y\in B$.
    It is a routine computation to verify that $r$ satisfies $r^4 = \id_{B\times B}$.
\end{itemize}
\end{exs}
\medskip

The previous examples are all regular q-cycle sets. We conclude the section by showing examples of q-cycle sets that are not regular.
\begin{exs}\label{ex:left_quasi-normal}\hspace{1mm}
\begin{itemize}
    \item[1)] 
    Let $X$ be a set, $\cdotp$ and $:$ two operations on $X$ given by
    \begin{align*}
    x\cdotp y:= y
    \qquad \qquad
    x : y:= k
    \end{align*}
    for all $x,y\in X$, where $k\in X$.  
    Then, it is easy to check that the structure $\left(X, \cdotp, :\right)$ is a q-cycle set. Clearly, if $|X| > 1$, then $\left(X, \cdotp, :\right)$ is not regular. Moreover, the left non-degenerate solution associated to $\left(X, \cdotp, :\right)$ is the map $r:X\times X\to X\times X$ given by
    \begin{align*}
        r\left(x,y\right) 
        = \left(y, k\right),
    \end{align*}
    for all $x,y\in X$. In particular, let us observe that $r$ satisfies the relation $r^3 = r^2$.
    \item[2)] Let $X$ be a left quasi-normal semigroup, i.e., $X$ is a semigroup such that $xyz = xzyz$, for all $x,y,z\in X$. Then, it is straightforward to check that $X$ endowed by the operations $\cdotp$ and $:$ defined by
    \begin{align*}
        x\cdotp y:= y
        \qquad \qquad
        x : y := yx,
    \end{align*}
    for all $x,y\in X$, is a q-cycle set.
    Clearly, $(X,\cdotp,:)$ in general is not regular. Note that the solution associated to $X$ is given by
    \begin{align*}
        r\left(x,y\right)
        = \left(y, xy\right),
    \end{align*}
    for all $x,y\in X$, and it coincides with that provided in \cite[Examples 6.2]{CaMaSt19x}. In particular, we have that r satisfies the property $r^5 = r^3$.
\end{itemize}
\end{exs}

\section{Non-degeneracy of q-cycle sets}
This section focuses on non-degeneracy of q-cycle sets.  Rump in \cite[Theorem 2]{rump2005decomposition} and  Jespers and Okni\'{n}ski in \cite[Corollary 2.3]{JeOk05}, using different languages, showed that every finite involutive left non-degenerate solution is non-degenerate. Recently, Ced\'o, Jespers and Verwimp asked if the same result is true for every finite bijective left non-degenerate solution.
\medskip

\textbf{Question.} \cite[Question 4.2]{cedo2019structure} Is every finite left non-degenerate solution  non-degenerate?
\medskip

A first aim is to use the theory of q-cycle sets to answer in the positive sense. Moreover, we give sufficient conditions to find a class of non-degenerate solutions that include properly the finite left non-degenerate ones.
\medskip

Initially, we provide the ``q-version'' of the proof of \cite[Theorem 2]{rump2005decomposition} provided recently by Jedlicka, Pilitowska, and Zamojska-Dzienio in \cite[Proposition 4.7]{JeP18}.\\
Hereinafter, for any regular q-cycle set $X$, we denote by $\mathcal{G}(X)$ the subgroup of $\Sym(X)$ given by
$$\mathcal{G}(X):=<\{\sigma_x|x\in X\}\cup \{\delta_x|x\in X\}> $$
and we call it the \emph{permutation group associated to $X$}.
\begin{lemma}\label{lemmaimp}
Let $(X,\cdotp,:)$ be a regular q-cycle set such that the associated permutation group $\mathcal{G}(X)$ is finite. Then the squaring maps $\mathfrak{q}$ and $\mathfrak{q'}$ are surjective.
\end{lemma}
\begin{proof}
At first, note that thanks to \eqref{ug2}, we have
\begin{equation}\label{cons}
(\sigma_x^{-1}(y): x):(\sigma_x^{-1}(y): x)=y:(x:x)
\end{equation}
for every $x,y\in X$. Now, since $\mathcal{G}(X)$ is finite, there exists a natural number $m$ such that $\delta_z^m(z)=z$, for every $z\in X$. If $m\in\{1,2\}$ then $\mathfrak{q'}$ is surjective by \eqref{cons}. 
If $m>2$ and $z\in X$, then
\begin{eqnarray}
z&=& \delta_z^m(z) \nonumber \\
&=&\delta_z^{m-2}(z:(z: z)) \nonumber \\
&=& \delta_z^{m-2}(\mathfrak{q'}(\sigma_z^{-1}(z): z))) \nonumber \\
&=& \delta_z^{m-3}(z:\mathfrak{q'}(\sigma_z^{-1}(z): z))) \nonumber
\end{eqnarray}
and applying repeatedly \eqref{cons}, we obtain that $\mathfrak{q'}(v)=z$, for some $v\in X$, hence $\mathfrak{q'}$ is surjective. In a similar way, one can show that $\mathfrak{q}$ is surjective, therefore the thesis follows.
\end{proof}

\begin{theor}\label{first}
Let $(X,\cdotp,:)$ be a finite regular q-cycle set. Then $X$ is non-degenerate.
\end{theor}
\begin{proof}
If $X$ is finite then so $\mathcal{G}(X)$ is, hence by the previous lemma $\mathfrak{q}$ and $\mathfrak{q'}$ are surjective. Again, since $X$ is finite, $\mathfrak{q}$ and $\mathfrak{q'}$ are also injective, hence the thesis.
\end{proof}

As a consequence of the previous theorem, we obtain the following result that extends the analogous one for finite involutive solutions showed by Rump in \cite[Theorem 2]{rump2005decomposition} and by Jespers and Okni\'{n}ski in \cite[Corollary 2.3]{JeOk05}. 
\begin{cor}
Let $(X,r)$ be a finite bijective left non-degenerate solution. Then $(X,r)$ is non-degenerate.
\end{cor}

\begin{proof}
It follows by the previous theorem and \cite[Proposition 1]{rump2019covering}.
\end{proof}
\medskip
Now, we use the previous result on finite solutions to find a larger class of non-degenerate solutions. At first, we have to introduce a preliminary lemma. We recall that a subset $Y$ of a q-cycle set $X$ is said to be $\mathcal{G}(X)$-invariant if $g\left(y\right)\in Y$, for all $g\in\mathcal{G}(X)$ and $y\in Y$.

\begin{lemma}\label{invar}
Let $X$ be a regular q-cycle set and $Y$ a $\mathcal{G}(X)$-invariant subset of $X$. 
Then, $Y$ is a regular q-cycle set with respect to the operations induced by $X$. Moreover, if $X$ is non-degenerate then so $Y$ is.
\end{lemma}

\begin{proof}
Is is a straightforward calculation.
\end{proof}

\begin{theor}
Let $X$ be a regular q-cycle set such that the orbits of $X$ with respect to the action of $\mathcal{G}(X)$ have finite size. Then, $X$ is non-degenerate.
\end{theor}

\begin{proof}
We have to show that the squaring maps $\mathfrak{q}$ and $\mathfrak{q'}$ are bijective. Let $x$ be an element of $X$ and $Y$ the orbit of $X$ such that $x\in Y$. 
By the previous lemma and the hypothesis, $Y$ is a finite regular q-cycle set and, by \cref{first}, it is non-degenerate. Therefore, there exist $x_1,x_2\in Y$ such that $\mathfrak{q}(x_1)=x$ and $\mathfrak{q'}(x_2)=x$. 
Hence, we have that the squaring maps of $X$ are surjective. Now, suppose that $y,z,t$ are elements of $X$ such that $\mathfrak{q}(y)=\mathfrak{q}(z)=t$ and let $T$ be the orbit of $X$ such that $t\in T$. Then, we have that $y=\sigma_y^{-1}(t)\in T$ and $z=\sigma_z^{-1}(t)\in T$. Since $T$ is non-degenerate, we obtain that $y=z$, hence the squaring map $\mathfrak{q}$ of $X$ is injective. In a similar way, one can show that the squaring map $\mathfrak{q'}$ of $X$ is injective.
\end{proof}

We conclude the section by translating the previous result in terms of solutions. Note that, if $X$ is a regular q-cycle set and $(X,r)$ the associated solution, where $r(x,y):=(\lambda_x(y),\rho_y(x))$, for all $x,y\in X$, then the associated permutation group $\mathcal{G}(X)$ coincides with the subgroup of $\Sym(X)$ generated by the set $\{\lambda_x|x\in X\}\cup\{\eta_x|x\in X\}$, where $\eta_x(y)=\rho_{\lambda^{-1}_y(x)}(y)$, for all $x,y \in X$.

\begin{cor}
Let $(X,r)$ be a bijective left non-degenerate solution such that the orbits of $X$ respect to the action of $\mathcal{G}(X)$ have finite size. Then $(X,r)$ is non-degenerate.
\end{cor}

\begin{proof}
It follows by the previous theorem and \cite[Proposition 1]{rump2019covering}.
\end{proof}

\section{Retraction of regular q-cycle sets}

In this section we introduce the ``q-version'' of the retract relation introduced by Rump in \cite{rump2005decomposition} for cycle sets. Namely, we introduce a suitable notion of retraction that is a congruence of q-cycle sets, i.e., it is a congruence with respect to the two operations $\cdotp$ and $:$. 
Furthermore, we show that, for the specific class of non-degenerate q-cycle sets, the quotient of a q-cycle set by its retraction is still a q-cycle set.

\begin{defin}
Let $(X,\cdotp,:)$ be a regular q-cycle set and $\sim$ the relation on $X$ given by
$$x\sim y:\Longleftrightarrow \sigma_x=\sigma_y 
\quad \mbox{and}\quad   
\delta_x=\delta_y$$
for every $x,y\in X$. Then $\sim$ will be called the \textit{retract relation}.
\end{defin}

In analogy to \cite[Lemma 2]{rump2005decomposition}, we show that $\sim$ is always a congruence of a q-cycle set.

\begin{prop}
Let $(X,\cdotp,:)$ be a regular q-cycle set. Then, the retraction $\sim$ is a congruence of $(X,\cdotp,:)$.
\end{prop}

\begin{proof}
Let $x,y,z,t$ be such that $x\sim y$ and $z\sim t$. Then, using \eqref{ug1},
\begin{eqnarray}
(x\cdotp z)\cdotp (x\cdotp k) &=& (y\cdotp z)\cdotp (y\cdotp k)=(z:y)\cdotp (z\cdotp k)=(t:y)\cdotp (t\cdotp k)\nonumber \\
&=& (y\cdotp t)\cdotp (y\cdotp k),\nonumber 
\end{eqnarray}
therefore $\sigma_{x\cdotp z}=\sigma_{y\cdotp t}$.
Moreover, using \eqref{ug2},
\begin{eqnarray}
(x: z)\cdotp (x\cdotp k) &=& (y: z)\cdotp (y\cdotp k) = (z\cdotp y)\cdotp (z\cdotp k)=(t\cdotp y)\cdotp (t\cdotp k)\nonumber \\
&=& (y: t)\cdotp (y\cdotp k),\nonumber 
\end{eqnarray}
therefore $\sigma_{x: z}=\sigma_{y: t}$. Similarly, one can show that $\delta_{x\cdotp z}=\delta_{y\cdotp t}$ and $\delta_{x: z}=\delta_{y: t}$, hence the thesis.
\end{proof}

%
%
\medskip

To prove the main result of this section we make use of the extension on the free group  $F(X)$ of a q-cycle set $X$. Initially, we recall the following proposition contained in \cite{rump2019covering}.
\begin{prop}\emph{\cite[Theorem 1]{rump2019covering}}\label{teorump}
Let $X$ be a non-degenerate regular q-cycle set. Then $X$ can be extended to a q-cycle set on the free group $(F(X),\circ)$ such that $1\cdotp a=1:a=a$ and
\begin{itemize}
\item[1)] $(a\circ b)\cdotp c=a\cdotp (b\cdotp c)$
\item[2)] $(a\circ b): c=a: (b: c)$
\item[3)] $a\cdotp(b\circ c)=((c:a)\cdotp b)\circ (a\cdotp c)$
\item[4)] $a:(b\circ c)=((c\cdotp a): b)\circ (a: c)$, 
\end{itemize}
are satisfied, for all $a,b,c\in F(X)$. Furthermore, in $F(X)$ we have that 
\begin{align*}
   x\cdotp y^{-1}=(\delta_y^{-1}(x)\cdotp y)^{-1} 
   \qquad
   \text{and}
   \qquad
   x: y^{-1}=(\sigma_y^{-1}(x): y)^{-1} 
\end{align*}
hold, for all $x,y\in X$.
\end{prop}
\medskip

As a consequence of the previous proposition we have the following result.
\begin{lemma}\label{prop-teorump}
Let $X$ be a non-degenerate regular q-cycle set. Then, the following hold:
\begin{enumerate}
    \item[1)] $a\cdotp x\in X$, for all $a\in F(X)$ and $x\in X$;
    \item[2)] $a:x\in X$, for all $a\in F(X)$ and $x\in X$;
    \item[3)] for all $x,x_1,\ldots,x_n\in X$ and $\epsilon,\epsilon_1,\ldots,\epsilon_n\in \lbrace -1, 1 \rbrace$, there exist $t_1,...,t_n\in X$ and $\eta, \eta_1,...,\eta_n\in \{1,-1\}$ such that
    \begin{align*}
        \sigma_{x^{\epsilon}}(x_1^{\epsilon_1}\circ\cdots \circ x_{n}^{\epsilon_{n}})=t_1^{\eta_1}\circ\cdots \circ t_n^{\eta_n};
    \end{align*}
    \item[4)] for all $x, x_1,\ldots,x_n\in X$ and $\epsilon,\epsilon_1,\ldots,\epsilon_n\in \lbrace -1, 1 \rbrace$, there exist $t_1,...,t_n\in X$ and $\eta, \eta_1,...,\eta_n,\epsilon\in \{1,-1\}$ such that
    \begin{align*}
        \delta_{x^{\epsilon}}(x_1^{\epsilon_1}\circ\cdots \circ x_{n}^{\epsilon_{n}})=t_1^{\eta_1}\circ\cdots \circ t_n^{\eta_n}.
    \end{align*}
\end{enumerate}
\begin{proof}
$1)$ \ If $x,y\in X$, by $1)$ in \cref{teorump}, it follows that
\begin{align*}
    \sigma_{x^{-1}}\sigma_{x}\left(y\right)
    = x^{-1}\cdotp\left(x\cdotp y\right)
    = \left(x^{-1}\circ x\right)\cdotp y
    = 1\cdotp y
    = y
\end{align*}
and, analogously, $\sigma_{x}\sigma_{x^{-1}}\left(y\right) = y$, i.e., $\sigma_{x^{-1}} = \sigma_{x}^{-1}$.
Thus, if $a\in F(X)$, there exist $x_1,\ldots, x_n\in X$ and $\epsilon_1,\ldots, \epsilon_n\in\lbrace-1,1\rbrace$ such that
$a = x_1^{\epsilon_1}\circ\cdots\circ x_n^{\epsilon_n}$, hence, by $1)$ in \cref{teorump}, it holds that
\begin{align*}
    a\cdotp x 
    = \sigma_{x_1^{\epsilon_1}}\ldots\sigma_{x_n^{\epsilon_n}}\left(x\right)
    = \sigma_{x_1}^{\epsilon_1}\ldots\sigma_{x_n}^{\epsilon_n}\left(x\right)\in X.
\end{align*}
$2)$ \ The proof is similar to $1)$.\\
$3)$ \ 
We proceed by induction on $n\in\mathbb{N}$. Let $n = 1$, $x, x_1\in X$, $\epsilon,\epsilon_1\in\lbrace -1,1 \rbrace$. Let us note that if $\epsilon_1 = 1$, then the thesis follows by $1)$. If $\epsilon_1 = -1$, 
since by \cref{teorump}, $a\cdotp 1 = 1$, for every $a\in F(X)$, by $3)$ in \cref{teorump} it follows that 
$1 = x^{\epsilon}\cdotp \left(x_1\circ x_1^{-1}\right) = \left(\left(x_1^{-1} : x^{\epsilon}\right)\cdotp x_1\right)\circ \left(x^{\epsilon}\cdotp x_1^{-1}\right)$ 
and so
\begin{align*} 
    \sigma_{x^{\epsilon}}\left(x_1^{-1}\right)
    = \left(\left(x_1^{-1} : x^{\epsilon}\right)\cdotp x_1\right)^{-1}
\end{align*}
with $\left(x_1^{-1} : x^{\epsilon}\right)\cdotp x_1\in X$ by $1)$.
Now, suppose the thesis holds for a natural number $n$ and let $x_1,\ldots,x_n\in X$, $\epsilon_1,\ldots,\epsilon_n\in\lbrace -1,1 \rbrace$. Then, by $3)$ in \cref{teorump}, we have that
\begin{align*}
    \sigma_{x^{\epsilon}}\left(x_1^{\epsilon_1}\circ\cdots\circ x_{n+1}^{\epsilon_{n+1}}\right)
    = \left(\left(x_{n+1}^{\epsilon_{n+1}} : x^{\epsilon} \right)\cdotp \left(x_1^{\epsilon_1}\circ \cdots \circ x_n^{\epsilon_n}\right)\right)
    \circ\left(x^{\epsilon}\cdotp x_{n+1}^{\epsilon_{n+1}}\right),
\end{align*}
therefore, by induction hypothesis, the thesis follows.\\
$4)$ \ The proof is similar to $3)$.
\end{proof}
\end{lemma}
\medskip

Now, let us introduce the following preliminary lemmas.
\begin{lemma}\label{lemsig}
Let $X$ be a non-degenerate regular q-cycle set and $x,y\in X$. Then, the following statements are equivalent
\begin{itemize}
\item[1)] $x\cdotp z=y\cdotp z$ and $x:z=y:z$, for every $z\in X$;
\item[2)] $x\cdotp z=y\cdotp z$ and $x:z=y:z$, for every $z\in F(X)$.
\end{itemize}
\end{lemma}

\begin{proof}
Since every element of $F(X)$ can be written as $x_1^{\epsilon_1}\circ...\circ x_n^{\epsilon_n}$ for some $x_1,...,x_n\in X$ and $\epsilon_1,...,\epsilon_n\in \{1,-1\}$ we prove ``$1)\Rightarrow 2)$'' by induction on $n$. Hence, suppose that $x\cdotp t=y\cdotp t$ and $x:t=y:t$, for every $t\in X$.\\
If $n=1$ and $z=x_1^{\epsilon_1}$ for some $x_1\in X,$ $\epsilon_1\in \{1,-1\}$ we have two cases: if $\epsilon_1=1$, then $x: z=y: z$ by the hypothesis, if $\epsilon_1=-1$, by equality \eqref{ug2} and the hypothesis we have that 
\begin{eqnarray}
& &x:(x_1:x_1)=y:(x_1:x_1)\Rightarrow  \nonumber \\
&\Rightarrow &(x_1\cdotp \sigma_{x_1}^{-1}(x)):(x_1:x_1)=(x_1\cdotp \sigma_{x_1}^{-1}(y)):(x_1:x_1) \nonumber \\
&\Rightarrow & (\sigma_{x_1}^{-1}(x):x_1):(\sigma_{x_1}^{-1}(x):x_1)=(\sigma_{x_1}^{-1}(y):x_1):(\sigma_{x_1}^{-1}(y):x_1)\nonumber \\
&\Rightarrow & \mathfrak{q'}(\sigma_{x_1}^{-1}(x):x_1)=\mathfrak{q'}(\sigma_{x_1}^{-1}(y):x_1) \nonumber \\
&\Rightarrow & \sigma_{x_1}^{-1}(x):x_1=\sigma_{x_1}^{-1}(y):x_1 \nonumber
\end{eqnarray} 
and, since $F(X)$ is a group, it follows that $(\sigma_{x_1}^{-1}(x):x_1)^{-1}=(\sigma_{x_1}^{-1}(y):x_1)^{-1}$ and hence, by \cref{teorump}, $x: x_1^{-1}=y: x_1^{-1}$. Similarly, one can show that $x\cdotp x_1^{-1}=y\cdotp x_1^{-1}$.\\
Now, suppose the thesis for a natural number $n$ and let $z:=x_1^{\epsilon_1}\circ...\circ x_{n+1}^{\epsilon_{n+1}}\in F(X)$, where $x_1,...,x_{n+1}\in X$ and $\epsilon_1,...,\epsilon_{n+1}\in \{1,-1\}$. Then, by $4)$ in \cref{teorump}, we have that 
\begin{eqnarray}
x:(x_1^{\epsilon_1}\circ...\circ x_{n+1}^{\epsilon_{n+1}}) &=& ((x_{n+1}^{\epsilon_{n+1}}\cdotp x):x_1^{\epsilon_1}\circ...\circ x_{n}^{\epsilon_{n}})\circ(x:x_{n+1}^{\epsilon_{n+1}}) \nonumber \\
&=& ((x_{n+1}^{\epsilon_{n+1}}\cdotp x):(x_{n+1}^{\epsilon_{n+1}}:\delta^{-1}_{x_{n+1}^{\epsilon_{n+1}}}(x_1^{\epsilon_1}\circ...\circ x_{n}^{\epsilon_{n}})))\circ(x:x_{n+1}^{\epsilon_{n+1}}) \nonumber \\
&=& ((x:x_{n+1}^{\epsilon_{n+1}}):(x:\delta^{-1}_{x_{n+1}^{\epsilon_{n+1}}}(x_1^{\epsilon_1}\circ...\circ x_{n}^{\epsilon_{n}})))\circ(x:x_{n+1}^{\epsilon_{n+1}}) \nonumber
\end{eqnarray}
and similarly
$$y:(x_1^{\epsilon_1}\circ...\circ x_{n+1}^{\epsilon_{n+1}})  =((y:x_{n+1}^{\epsilon_{n+1}}):(y:\delta^{-1}_{x_{n+1}^{\epsilon_{n+1}}}(x_1^{\epsilon_1}\circ...\circ x_{n}^{\epsilon_{n}})))\circ(y:x_{n+1}^{\epsilon_{n+1}}). $$
Note that, by \cref{prop-teorump}, there exist $t_1,...,t_n\in X$ and $\eta_1,...,\eta_n\in \{1,-1\}$ such that 
$\delta^{-1}_{x_{n+1}^{\epsilon_{n+1}}}(x_1^{\epsilon_1}\circ...\circ x_{n}^{\epsilon_{n}})=t_1^{\eta_1}\circ...\circ t_n^{\eta_n}$,
and this fact, together with the inductive hypothesis, implies that 
$$x:\delta^{-1}_{x_{n+1}^{\epsilon_{n+1}}}(x_1^{\epsilon_1}\circ...\circ x_{n}^{\epsilon_{n}}))=y:\delta^{-1}_{x_{n+1}^{\epsilon_{n+1}}}(x_1^{\epsilon_1}\circ...\circ x_{n}^{\epsilon_{n}})).$$
Moreover, again by inductive hypothesis, it follows that $x:x_{n+1}^{\epsilon_{n+1}}=y:x_{n+1}^{\epsilon_{n+1}}$. 
Therefore, we showed that $x:z=y:z$. In a similar way, one can show that $x\cdotp z=y\cdotp z$, hence ``$1)\Rightarrow 2)$''. The converse implication is trivial.
\end{proof}

Let $X$ be a regular non-degenerate q-cycle set and $F(X)$ its extension to the free group on $X$. In \cite[Section 4]{rump2019covering} Rump define the \emph{socle} of $F(X)$, denoted by $Soc(F(X))$, as the set
$$Soc(F(X)):=\{a \hspace{1mm} |\hspace{1mm} a\in F(X),\hspace{1mm} a\cdotp b=a:b=b\hspace{3mm}\forall\hspace{1mm} b\in F(X) \} $$
and he showed that $Soc(F(X)) $ is a subgroup of $F(X)$ and that the factor $F(X)/Soc(F(X)) $ is again a non-degenerate q-cycle set. From this fact the following lemma follows.
\begin{lemma}
Let $X$ be a non-degenerate q-cycle set and $F(X)$ its extension to the free group on $X$. Then, $\Ret(F(X))$ is a regular non-degenerate q-cycle set.
\end{lemma}

\begin{proof}
It is easy to see that $x\sim y$ if and only if $x\circ y^{-1}\in Soc(F(X))$ for every $x,y\in F(X)$, hence the result follows by the previous remark.
\end{proof}

\begin{theor}
Let $X$ be a regular q-cycle set and $\Ret(X)$ its retraction. Then, $X$ is non-degenerate if and only if $\Ret(X)$ is a non-degenerate q-cycle set. 
\end{theor}

\begin{proof}
To avoid confusion, for any element $x\in X$ we denote by $\bar{x}$ its equivalence class with respect to the retract relation on $X$ and by $\tilde{x}$ its equivalence class with respect to the retract relation on $F(X)$.\\
Suppose that $X$ is non-degenerate.  Then, the function 
$$G:X\longrightarrow \Ret(F(X)), \ x\mapsto  \tilde{x}$$ 
is a homomorphism of q-cycle sets. Moreover, by \cref{lemsig}, we have that $G(x)=G(y)$ if and only if $\bar{x}=\bar{y}$, hence $\Ret(X)\cong G(X)$ as algebraic structures.\\
If $z\in X$ then 
$$\tilde{z}\cdotp \tilde{x}=\widetilde{z\cdotp x}\qquad \tilde{z}: \tilde{x}=\widetilde{z: x} $$

$$\widetilde{z^{-1}}\cdotp \tilde{x} = \widetilde{z^{-1}\cdotp x} = \widetilde{\sigma_z^{-1}(x)}$$ and 
$$\widetilde{z^{-1}}: \tilde{x}=\widetilde{z^{-1}: x} = \widetilde{\delta_z^{-1}(x)}$$
and these equalities imply that $G(X)$ is a $\mathcal{G}(\Ret(F(X)))$-invariant subset of $\Ret(F(X))$, hence, by Lemma \ref{invar}, it is a non-degenerate q-cycle set. Therefore, $\Ret(X)$ is a non-degenerate q-cycle set.\\
Conversely, suppose that $\Ret(X)$ is a non-degenerate q-cycle set. We have to show that the maps $\mathfrak{q}$ and $\mathfrak{q'}$ of $X$ are bijective. Let $x,y\in X$ such that $x\cdotp x=y\cdotp y$. Then, $\bar{x}\cdotp \bar{x}=\bar{y}\cdotp \bar{y}$ and, by the non-degeneracy of $\Ret(X)$, we have that $\bar{x}=\bar{y}$. Hence, $y\cdotp x=x\cdotp x=y\cdotp y$, therefore $x=y$ and the injectivity of $\mathfrak{q}$ follows. 
Moreover, if $x\in X$, there exists a unique $\bar{y}\in\Ret(X)$ such that $\bar{y}\cdotp\bar{y}=\bar{x}$. 
Since $\bar{x}=\bar{y}\cdotp \overline{\sigma^{-1}_y(x)}$, we have that $\overline{\sigma^{-1}_y(x)}=\bar{y}$, and hence
$$x=\sigma_y(\sigma_y^{-1}(x))=\sigma_{\sigma^{-1}_y(x)}(\sigma_y^{-1}(x))=\mathfrak{q}(\sigma_y^{-1}(x)) $$
therefore $\mathfrak{q}$ is bijective. In the same way, one can show that $\mathfrak{q'}$ is bijective.
\end{proof}
\medskip

If $X$ is a non-degenerate q-cycle set and $(X,r)$ is the associated solution, it is easy to see that $x\sim y$ if and only if $\lambda_x=\lambda_y$ and $\rho_x=\rho_y$. Indeed, if $x,y\in X$, 
\begin{eqnarray}
x\sim y &\Leftrightarrow & \sigma_x=\sigma_y \quad \text{and}\quad \delta_x=\delta_y \nonumber \\
&\Leftrightarrow & \sigma^{-1}_x=\sigma^{-1}_y \quad \text{and}\quad \delta_x\mathfrak{q}'=\delta_y\mathfrak{q}' \nonumber \\
&\Leftrightarrow & \sigma^{-1}_x=\sigma^{-1}_y \quad \text{and}\quad \mathfrak{q}'(\sigma_z^{-1}(x):z)=\mathfrak{q}'(\sigma_z^{-1}(y):z)\quad \forall z\in X \nonumber \\
&\Leftrightarrow & \lambda_x=\lambda_y \quad \text{and} \quad \rho_x=\rho_y .\nonumber
\end{eqnarray}
Hence, the previous theorem give us an alternative proof of \cite[Theorem 3.3]{JeP18}, in terms of solutions $(X,r)$ having bijective map $r$. Moreover, it shows a kind of converse for left non-degenerate bijective solutions: indeed, if $X$ retracts to a non-degenerate solution then so $X$ is.

\begin{cor}
Let $(X,r)$ be a bijective left non-degenerate solution of the Yang-Baxter equation and $\Ret(X) $ the retraction of the associated q-cycle set. If $(X,r)$ is non-degenerate then, for every $\bar{x}\in \Ret(X)$, we can define the function $\bar{\lambda}_{\bar{x}}$ by
$$\bar{\lambda}_{\bar{x}}:\Ret(X)\longrightarrow \Ret(X),\ \bar{y}\mapsto \overline{\lambda_x(y)}$$ 
for every $\bar{y}\in \Ret(X)$. Moreover, the pair $(\Ret(X),\bar{r})$, where $\bar{r}$ is given by 
$$\bar{r}(\bar{x},\bar{y}):=(\overline{\lambda_x(y)},\overline{\lambda_x(y):x}) $$
is a bijective non-degenerate solution of the Yang-Baxter equation.\\
Conversely, suppose that the function $\bar{\lambda}_{\bar{x}}$ is well defined for every $\bar{x}\in \Ret(X)$ and that the pair $(\Ret(X),\bar{r})$ defined as above is a non-degenerate bijective solution. Then, $(X,r)$ is non-degenerate.
\end{cor}

\section{Dynamical extensions of q-cycle sets}

Inspired by the extensions of cycle sets and racks, introduced by Vendramin in \cite{vendramin2016extensions} and 
Andruskiewitsch and Gra\~{n}a in \cite{andruskiewitsch2003racks},
and the dynamical cocycles of biracks \cite{nelson2013birack}, introduced by Nelson and Watterberg, in this section we develop a theory of dynamical extensions of q-cycle sets.

\begin{theor}\label{lemm1}
Let $(X,\cdotp,:)$ be a q-cycle set, $S$ a set, $\alpha:X\times X\times S\longrightarrow\Sym(S)$ and $\alpha':X\times X\times S\longrightarrow S^S$ maps such that
\begin{equation}\label{ugd1}
\alpha_{(x\cdotp y),(x\cdotp z)}(\alpha_{(x,y)}(s,t),\alpha_{(x,z)}(s,u))=\alpha_{(y:x),(y\cdotp z)}(\alpha'_{(y,x)}(t,s),\alpha_{(y,z)}(t,u))
\end{equation}
\begin{equation}\label{ugd2}
\alpha'_{(x: y),(x: z)}(\alpha'_{(x,y)}(s,t),\alpha'_{(x,z)}(s,u))=\alpha'_{(y\cdotp x),(y:z)}(\alpha_{(y,x)}(t,s),\alpha'_{(y,z)}(t,u))
\end{equation}
\begin{equation}\label{ugd3}
\alpha'_{(x\cdotp y),(x\cdotp z)}(\alpha_{(x,y)}(s,t),\alpha_{(x,z)}(s,u))=\alpha_{(y: x),(y: z)}(\alpha'_{(y,x)}(t,s),\alpha'_{(y,z)}(t,u))
\end{equation}
hold, for all $x,y,z\in X$, $s,t,u\in S$. Then, the triple $(X\times S,\cdotp,:)$ where
$$(x,s)\cdotp (y,t):=(x\cdotp y,\alpha_{(x,y)}(s,t)) $$
$$(x,s): (y,t):=(x: y,\alpha'_{(x,y)}(s,t)), $$
for all $x,y\in X$ and $s,t\in S$, is a q-cycle set. Moreover, $(X\times S,\cdotp,:)$ is regular if and only if $\alpha'_{(x,y)}(s,-)\in\Sym(S)$, for all $x,y\in X$ and $s\in S$.
\end{theor} 

\begin{proof}
Since 
$$((x,s)\cdotp (y,t))\cdotp((x,s)\cdotp (z,u))=((x\cdotp y)\cdotp (x\cdotp z),\alpha_{(x\cdotp y),(x\cdotp z)}(\alpha_{(x,y)}(s,t),\alpha_{(x,z)}(s,u))) $$
and 
$$((y,t): (x,s))\cdotp((y,t)\cdotp (z,u)) =((y:x)\cdotp (y\cdotp z),\alpha_{(y:x),(y\cdotp z)}(\alpha'_{(y,x)}(t,s),\alpha_{(y,z)}(t,u))),$$
for all $x,y,z\in X$, $s,t,u\in S$, we have that \eqref{ugd1} is equivalent to \eqref{ug1}. In the same way one can show that \eqref{ugd2} and \eqref{ugd3} are equivalent to \eqref{ug2} and \eqref{ug3}. The rest of the proof is a straightforward calculation.
\end{proof}

\begin{defin}
We call the pair $(\alpha,\alpha')$ \textit{dynamical pair} and we call the q-cycle set $X\times_{\alpha,\alpha'} S:=(X\times S,\cdotp,:)$ \textit{dynamical extension} of $X$ by $S$.
\end{defin}

Clearly, every dynamical extension of a cycle set \cite{vendramin2016extensions} is a dynamical extension of a q-cycle set for which $\alpha=\alpha'$.\\
If $X$ and $S$ are q-cycle sets and
\begin{align*}
    \alpha_{(x,y)}(s,t) = \alpha'_{(x,y)}(s,t) := t
\end{align*}
for all $x,y\in X$, $s,t\in S$, $X\times_{\alpha,\alpha'} S$, is a dynamical extension which we call \textit{trivial dynamical extension} of $X$ by $S$. The goal of the next section is to provide further examples of dynamical extensions.

\begin{defin}
Let $X$ and $Y$ be q-cycle sets. A homomorphism $p:X\longrightarrow Y$ is called \textit{covering map} if it is surjective and all the fibers 
\begin{align*}
   p^{-1}(y):=\{x|x\in X,\hspace{1mm} p(x)=y \} 
\end{align*}
have the same cardinality. Moreover, a q-cycle set $X$ is \textit{simple} if, for every covering map $p:X\longrightarrow Y$, we have that $|Y|=1$ or $p$ is an isomorphism.
\end{defin}

\begin{exs}\hspace{1mm}
\begin{itemize}
\item[1)] Every q-cycle set having a prime number of elements is simple.
\item[2)] Let $X:=\{0,1,2,3\}$ be the q-cycle set given by
$$\sigma_0:=(1\quad 3)\qquad \sigma_1:=(0\quad3)\qquad \sigma_2:=(0\quad 1 \quad 3)\qquad \sigma_3:=(0\quad1) $$
$$\delta_0=\delta_1=\delta_3:=id_X\qquad \delta_2:=(0\quad 1\quad 3).$$
Suppose that $p:X\longrightarrow Y$ is a covering map. Hence, $Y$ has necessarily two elements. Moreover, since $p(0)\cdotp p(0) = p(0):p(0) = p(0)$, $Y$ is the q-cycle set given by $\sigma_y = \delta_y= \id_Y$, for every $y\in Y$.
This fact implies that $X$ has a subset $I$ of $2$ elements  such that $\sigma_x(I)=\delta_x(I)=I$, for every $x\in X$, but this is a contradiction. Then, $X$ is simple.
\end{itemize}
\end{exs}
\medskip

The following theorem, that is a ``q-version'' of \cite[Theorem 2.12]{vendramin2016extensions}, states that every non-simple q-cycle set can be found as a dynamical extension of a smaller q-cycle set.

\begin{theor}
Let $X$ and $Y$ be q-cycle sets and suppose that $p:Y\longrightarrow X$ is a covering map. Then, there exist a set $S$ and a dynamical pair $(\alpha,\alpha')$ such that $Y\cong X\times_{\alpha,\alpha'} S$.
\end{theor}

\begin{proof}
Since all the fibers $p^{-1}(x)$ are equipotent, there exist a set $S$ and a bijection $f_x:p^{-1}(x)\longrightarrow S$. Let $\alpha:X\times X\times S\longrightarrow\Sym(S)$ and $\alpha':X\times X\times S\longrightarrow S^S$ given by
$$ \alpha_{(x,z)}(s,t):=f_{x\cdotp z}(f^{-1}_x(s)\cdotp f^{-1}_z(t))$$
and
$$ \alpha'_{(x,z)}(s,t):=f_{x: z}(f^{-1}_x(s): f^{-1}_z(t))$$
for all $x,z\in X$, $s,t\in S$. Then,
\begin{eqnarray}
 &&\alpha_{(x\cdot y,x\cdot z)}(\alpha_{(x,y)}(r,s),\alpha_{(x,z)}(r,t))\quad= \nonumber \\
&=& f_{(x\cdotp y)\cdotp (x\cdotp z)}(f^{-1}_{x\cdotp y}(\alpha_{(x,y)}(r,s))\cdotp f^{-1}_{x\cdotp z}(\alpha_{(x,z)}(r,t)))\nonumber \\
&=& f_{(x\cdotp y)\cdotp (x\cdotp z)}(f^{-1}_{x\cdotp y}(f_{x\cdotp y}(f^{-1}_x(r)\cdotp f^{-1}_y(s)) \cdotp f^{-1}_{x\cdotp z}(f_{x\cdotp z}(f^{-1}_x(r)\cdotp f^{-1}_z(t))\nonumber \\
&=& f_{(x\cdotp y)\cdotp (x\cdotp z)}((f^{-1}_x(r)\cdotp f^{-1}_y(s))\cdotp (f^{-1}_x(r)\cdotp f^{-1}_z(t))\nonumber \\
&=& f_{(y: x)\cdotp (y\cdotp z)}((f^{-1}_y(s): f^{-1}_x(r))\cdotp (f^{-1}_y(s)\cdotp f^{-1}_z(t))\nonumber \\
&=& f_{(y: x)\cdotp (y\cdotp z)}(f^{-1}_{y: x}(f_{y: x}(f^{-1}_y(s): f^{-1}_x(r)) \cdotp f^{-1}_{y\cdotp z}(f_{y\cdotp z}(f^{-1}_y(s)\cdotp f^{-1}_z(t))\nonumber \\
&=& \alpha_{(y: x,y\cdot z)}(\alpha'_{(y,x)}(s,r),\alpha_{(y,z)}(s,t)),\nonumber
\end{eqnarray}
for all $x,y,z\in X$ and $r,s,t\in S$, hence \eqref{ugd1} holds. In the same way, one can check that the pair $(\alpha,\alpha')$ verifies \eqref{ugd2} and \eqref{ugd3}, hence, by \cref{lemm1}, $X\times_{\alpha,\alpha'} S$ is a q-cycle set.\\
The map $\phi:Y\longrightarrow X\times_{\alpha,\alpha'} S$, $y\mapsto (p(y),f_{p(y)}(y))$ is a bijective homomorphism. Indeed, 
\begin{eqnarray}
\phi(y\cdotp z)&=&(p(y\cdotp z),f_{p(y\cdotp z)}(y\cdotp z)) \nonumber \\
&=& (p(y)\cdotp p(z),f_{p(y)\cdotp p(z)}(f^{-1}_{p(y)}(f_{p(y)}(y))\cdotp f^{-1}_{p(z)}(f_{p(z)}(z)))) \nonumber \\
&=& (p(y)\cdotp p(z),\alpha_{(p(y),p(z))}(f_{p(y)}(y),f_{p(z)}(z)))\nonumber \\
&=& (p(y),f_{p(y)}(y))\cdotp (p(z),f_{p(z)}(z))\nonumber \\
&=& \phi(y)\cdotp \phi(z),\nonumber
\end{eqnarray}
for all $y,z\in Y$. In a similar way one can show that $\phi(y: z)=\phi(y):\phi(z)$, for all $y,z\in Y$. By the surjectivity of $p$ and the bijectivity of the maps $f_y$ the bijectivity of $\phi$ also follows, hence the thesis.
\end{proof}

\section{Examples and constructions}
In this section, we collect several both known and new examples of dynamical extension of q-cycle sets. Finally, we introduce a new construction of q-cycle sets, the \emph{semidirect product}. 
\medskip

Some dynamical cocycles of biracks founded in \cite{bardakov2019general,nelson2013birack,horvat2018constructing} provide dynamical extensions of q-cycle sets (one can see this fact by a long but easy calculation). Using non-degenerate cycle sets, in \cite{vendramin2016extensions,cacsp2017,bachiller2015family} several examples of dynamical extensions were obtained. 

We start by giving an example of a dynamical extension of a degenerate cycle set that is similar to those obtained in \cite{cacsp2017} and \cite{bachiller2015family}.

\begin{ex}
Let $X$ be the cycle set on $\mathbb{Z}$ given by 
$x\cdotp y:=y-min\{0,x\}$,
for all $x,y\in X$, (see \cite[Example 1]{rump2005decomposition}), and let $S$ be the Klein group $\mathbb{Z}/2\mathbb{Z}\times\mathbb{Z}/2\mathbb{Z}$. Let $\alpha:X\times X\times S\longrightarrow \Sym(S)$ be the function given by
$$
\alpha_{(i,j)}((a,b),(c,d)):=
\begin{cases}
(c,d-(a-c))&\mbox{ if }i = j,\\
(c-b,d)&\mbox{ if }i\neq j
\end{cases}
$$
for all $i,j\in X$, $(a,b),(c,d)\in S$, and set $\alpha':=\alpha$. Then, $X\times_{\alpha,\alpha'}S$ is a dynamical extension of $X$ by $S$.
\end{ex}

By an easy calculation one can see that the previous example is an irretractable cycle set. Moreover, it is degenerate: indeed, $\mathfrak{q}(-2,0,0)$ is equal to $\mathfrak{q}(-1,0,0)$. Recently, Ced\'o, Jespers and Verwimp \cite{cedo2019structure} studied the links between non-degeneracy and irretractability of cycle sets (note that they used the language of set-theoretic solutions). As mentioned in \cite[p. 12]{cedo2019structure}, they did not know whether there exist  examples of irretractable degenerate cycle sets: in this direction, the previous example answers in the positive sense.
\medskip

Inspired by the dynamical extensions provided in \cite{bachiller2015family}, in the following example we give a family of dynamical extensions that determine new q-cycle sets that are not cycle sets.
\begin{ex}
Let $X$ be the cycle set given by $x\cdotp y=y$ for every $x,y\in X$, and $G$ an abelian group. Let $\alpha,\alpha':X\times X\times \left(G\times G\right)\longrightarrow \Sym(G\times G)$ be the maps given by
$$
\alpha_{(x,y)}((s_1,s_2),(t_1,t_2)):=
\begin{cases}
(t_1+t_2-s_2,t_2) &\mbox{ if } x=y\\
(t_1,t_2+s_1)     &\mbox{ if } x\neq y
\end{cases}
$$
and
$$
\alpha'_{(x,y)}((s_1,s_2),(t_1,t_2)):=
\begin{cases}
(t_1-t_2+s_2,t_2) &\mbox{ if } x=y\\
(t_1,t_2+s_1)    &\mbox{ if } x\neq y
\end{cases}
$$
for all $x,y\in X$, $s_1,s_2,t_1,t_2\in G$. Then $(\alpha,\alpha')$ is a dynamical pair and hence $X\times_{\alpha,\alpha'}(G\times G)$ is a q-cycle set.
\end{ex}

\begin{ex}
    Let $B$ be a group, $f$ an endomorphism of $B$ and $\left(B, \cdotp, :\right)$ the q-cycle in \cref{ex:semi-brace} where $x\cdotp y = x^{-1}yf\left(x\right)$ and $x:y = f\left(y\right)$, for all $x,y\in B$.
    Let $S = B$, $\alpha:B\times B\times S\to \Sym\left(S\right)$ and $\alpha':B\times B\times S\to S^S$ be the maps defined by 
\begin{align*}
    \alpha_{(x,y)}(s,t)  &:= x^{-1}t
    \\
    \alpha'_{(x,y)}(s,t) &:= f\left(t\right),
\end{align*}
for all $x,y,s,t\in B$. Then, it is a routine computation to verify that $\left(\alpha, \alpha'\right)$ is a dynamical pair. 
Thus, the dynamical extension $B\times_{\alpha, \alpha'} S$ of $B$ by $S$ is such that
\begin{align*}
    (x,s)\cdotp (y,t)&= (x^{-1}yf(x), \  x^{-1}t)\\
    (x,s): (y,t) &= (f(y), \ f(t)),
\end{align*}
for all $x,y,s,t\in B$. 
Let us observe that the q-cycle set $B\times_{\alpha, \alpha'} S$ is regular if and only if $f$ is bijective.
Moreover, the left non-degenerate solution associated to $B\times_{\alpha, \alpha'} S$ is given by
\begin{align*}
    r\left(\left(x,s\right), \left(y,t\right)\right)
    = \left(\left(xyf(x)^{-1}, \ xt\right),  \left(f(x), \ f(s)\right)\right),
\end{align*}
for all $x,y,s,t\in B$.
\end{ex}

\begin{ex}
Let $X$ be a left quasi-normal semigroup and $\left(X, \cdotp, :\right)$ the q-cycle set in $2)$ of \cref{ex:left_quasi-normal} where $x\cdotp y= y$ and $x : y = yx$, for all $x,y\in X$. Let $S = X$, $\alpha:X\times X\times S\to \Sym\left(S\right)$ and $\alpha':X\times X\times S\to S^S$ be the maps defined by 
\begin{align*}
    \alpha_{(x,y)}(s,t)  &= t 
    \\
    \alpha'_{(x,y)}(s,t) &= tx,
\end{align*}
for all $x,y,s,t\in X$. Then, $\left(\alpha, \alpha'\right)$ is a dynamical pair. In fact, clearly \eqref{ugd1} is satisfied. Moreover, if $x,y,s,t\in X$, we have that
\begin{align*}
&\alpha'_{(x: y),(x: z)}(\alpha'_{(x,y)}(s,t),\alpha'_{(x,z)}(s,u))
=  \alpha'_{(yx,zx)}(tx, ux)
= uxyx 
= uyx\\
&= \alpha'_{(x,zy)}(s ,uy)
= \alpha'_{(y\cdotp x),(y:z)}(\alpha_{(y,x)}(t,s),\alpha'_{(y,z)}(t,u))
\end{align*}
and
\begin{align*}
    &\alpha'_{(x\cdotp y),(x\cdotp z)}(\alpha_{(x,y)}(s,t),\alpha_{(x,z)}(s,u))
    = \alpha'_{(y,z)}(t,u) = uy
    = \alpha'_{(y,z)}(t,u)\\
    &= \alpha_{(y: x),(y: z)}(\alpha'_{(y,x)}(t,s),\alpha'_{(y,z)}(t,u)),
\end{align*}
i.e., \eqref{ugd2} and \eqref{ugd3} hold. 
Therefore, the dynamical extension $X\times_{\alpha, \alpha'} S$ of $X$ by $S$ is such that
\begin{align*}
    (x,s)\cdotp (y,t)&= (y, \ t)\\
    (x,s): (y,t) &= (yx, \ tx),
\end{align*}
for all $x,y,s,t\in X$. Note that $X\times_{\alpha, \alpha'} S$ in general is not regular.
Moreover, the left non-degenerate solution associated to the q-cycle set $X\times_{\alpha, \alpha'} S$ is given by
\begin{align*}
    r\left(\left(x,s\right), \left(y,t\right)\right)
    = \left(\left(y,t\right), \left(xy,sy\right)\right).
\end{align*}
Let us observe that this solution satisfies the property $r^5=r^3$, just like for the solution associated to the q-cycle set $\left(X,\cdotp,:\right)$. 
\end{ex}

The following example is the analogue of the abelian extension of cycle sets studied by Lebed and Vendramin \cite{lebed2017homology}. 

\begin{ex}
Let $X$ be a cycle set, $A$ an abelian group, $a,b,a',b'\in A$ and $f,f':X\times X\longrightarrow A$ given by 
$$
f(x,y):=
\begin{cases}
a &\mbox{ if } x=y\\
b &\mbox{ if } x\neq y
\end{cases}
$$
and
$$
f'(x,y):=
\begin{cases}
a' &\mbox{ if } x=y\\
b' &\mbox{ if } x\neq y
\end{cases}
$$
for all $x,y\in X$. Define on $X\times A$ the function $\alpha,\alpha':X\times X\times A\longrightarrow\Sym(A)$ given by
$$\alpha_{(x,y)}(s,t)=t+f(x,y) $$
and
$$\alpha'_{(x,y)}(s,t)=t+f'(x,y), $$
for all $x,y\in X$. Then, $(\alpha,\alpha')$ is a dynamical pair and hence $X\times_{\alpha,\alpha'} A$ is a q-cycle set. In particular, we have that
$$(x,s)\cdotp (y,t)=(x\cdotp y,t+f(x,y)) $$
$$(x,s): (y,t)=(x\cdotp y,t+f'(x,y)), $$
for all $x,y\in X$. 
\end{ex}

A particular family of dynamical extensions allow us to define a kind of semidirect product of q-cycle sets. Since this construction is similar to the ones of other algebraic structures, for the convenience of the reader we write the two operations $\cdotp$ and $:$ of the q-cycle set without using the dynamical pair. 

\begin{prop}
Let $X,S$ be q-cycle sets, $\theta:X\longrightarrow \Aut(S), x\mapsto \theta_x$ such that 
\begin{align*}
    \theta_{x\cdotp y}\theta_x =\theta_{y:x}\theta_{y},
\end{align*}
for all $x,y\in X$. Define on $X\times S$ the operations $\cdotp$ and $:$ by
$$(x,s)\cdotp (y,t):=(x\cdotp y,\theta_{x\cdotp y}(s)\cdotp \theta_{y: x}(t)) $$
$$(x,s): (y,t):=(x: y,\theta_{x: y}(s): \theta_{y\cdotp x}(t)) $$
for all $x,y\in X$, $s,t\in S$. Then, $(X\times S,\cdotp,:)$ is a q-cycle set which we call the \textit{semidirect product} of $X$ and $S$.
\end{prop}

\begin{proof}
Let $(x,s),(y,t),(z,u)\in X\times S$. Then
\begin{eqnarray}
& &((x,s)\cdotp (y,t))\cdotp ((x,s)\cdotp (z,u)) \nonumber \\
&=& (x\cdotp y,\theta_{x\cdotp y}(s)\cdotp \theta_{y: x}(t))\cdotp (x\cdotp z,\theta_{x\cdotp z}(s)\cdotp \theta_{z: x}(u)) \nonumber \\
&=& ((x\cdotp y)\cdotp (x\cdotp z),\theta_{(x\cdotp y)\cdotp (x\cdotp z)}(\theta_{x\cdotp y}(s)\cdotp \theta_{y: x}(t))\cdotp \theta_{(x\cdotp z):(x\cdotp y)}(\theta_{x\cdotp z}(s)\cdotp \theta_{z: x}(u)) \nonumber 
\end{eqnarray}
and, thanks to the hypothesis, we obtain that 
{\small\begin{align}\label{eq:1}
    &\theta_{(x\cdotp y)\cdotp (x\cdotp z)}(\theta_{x\cdotp y}(s)\cdotp \theta_{y: x}(t))\cdotp \theta_{(x\cdotp z):(x\cdotp y)}(\theta_{x\cdotp z}(s)\cdotp \theta_{z: x}(u))\\  &= (\theta_{(y:x)\cdotp (y\cdotp z)}(\theta_{y: x}(t)): \theta_{(x\cdotp y)\cdotp (x\cdotp z)}(\theta_{x\cdotp y}(s)))\cdotp (\theta_{(y:x)\cdotp (y\cdotp z)}(\theta_{y: x}(t))\cdotp \theta_{(z:y):(z:x)}(\theta_{z: y}(u)))\nonumber
\end{align}}
and
{\small\begin{align}\label{eq:2}
&\theta_{(y:x)\cdotp (y\cdotp z)}(\theta_{y:x}(t): \theta_{x\cdotp y}(s)))\cdotp \theta_{(y\cdotp z): (y:x)}(\theta_{y\cdotp z}(t)\cdotp \theta_{z: y}(u)\\  
&= (\theta_{(y:x)\cdotp (y\cdotp z)}(\theta_{y:x}(t)):\theta_{(x\cdotp y)\cdotp (x\cdotp z)}(\theta_{x\cdotp y}(s)))\cdotp (\theta_{(y:x)\cdotp (y\cdotp z)}(\theta_{(y:x)}(t))\cdotp \theta_{(z:y): (z:x)}(\theta_{z: y}(u))).\nonumber
\end{align}}
Indeed, 
{\small\begin{eqnarray*}
& &\theta_{(x\cdotp y)\cdotp (x\cdotp z)}(\theta_{x\cdotp y}(s)\cdotp \theta_{y: x}(t))\cdotp \theta_{(x\cdotp z):(x\cdotp y)}(\theta_{x\cdotp z}(s)\cdotp \theta_{z: x}(u)  \\
&=& (\theta_{(x\cdotp y)\cdotp (x\cdotp z)}(\theta_{x\cdotp y}(s))\cdotp \theta_{(x\cdotp y)\cdotp (x\cdotp z)}(\theta_{y: x}(t)))\cdotp (\theta_{(x\cdotp z):(x\cdotp y)}(\theta_{x\cdotp z}(s))\cdotp \theta_{(x\cdotp z):(x\cdotp y)}(\theta_{z: x}(u))) \nonumber \\
&=& (\theta_{(x\cdotp y)\cdotp (x\cdotp z)}(\theta_{x\cdotp y}(s))\cdotp \theta_{(x\cdotp y)\cdotp (x\cdotp z)}(\theta_{y: x}(t)))\cdotp (\theta_{(x\cdotp y)\cdotp (x\cdotp z)}(\theta_{x\cdotp y}(s))\cdotp \theta_{(x\cdotp z):(x\cdotp y)}(\theta_{z: x}(u))) \nonumber \\
&=& (\theta_{(x\cdotp y)\cdotp (x\cdotp z)}(\theta_{y: x}(t)): \theta_{(x\cdotp y)\cdotp (x\cdotp z)}(\theta_{x\cdotp y}(s)))\cdotp (\theta_{(x\cdotp y)\cdotp (x\cdotp z)}(\theta_{y: x}(t))\cdotp \theta_{(x\cdotp z):(x\cdotp y)}(\theta_{z: x}(u))) \nonumber \\
&=& (\theta_{(y:x)\cdotp (y\cdotp z)}(\theta_{y: x}(t)): \theta_{(x\cdotp y)\cdotp (x\cdotp z)}(\theta_{x\cdotp y}(s)))\cdotp (\theta_{(y:x)\cdotp (y\cdotp z)}(\theta_{y: x}(t))\cdotp \theta_{(z:x)\cdotp(z:y)}(\theta_{z: x}(u))) \nonumber \\
&=& (\theta_{(y:x)\cdotp (y\cdotp z)}(\theta_{y: x}(t)): \theta_{(x\cdotp y)\cdotp (x\cdotp z)}(\theta_{x\cdotp y}(s)))\cdotp (\theta_{(y:x)\cdotp (y\cdotp z)}(\theta_{y: x}(t))\cdotp \theta_{(z:y):(z:x)}(\theta_{z: y}(u))),\nonumber
\end{eqnarray*}}
i.e., \eqref{eq:1} holds. On the other hand, we have that 
{\small\begin{eqnarray}
& &((y,t): (x,s))\cdotp ((y,t)\cdotp (z,u))\nonumber   \\
&=& (y:x,\theta_{y:x}(t): \theta_{x\cdotp y}(s))\cdotp (y\cdotp z,\theta_{y\cdotp z}(t)\cdotp \theta_{z: y}(u)) \nonumber \\
&=& ((y:x)\cdotp (y\cdotp z),\theta_{(y:x)\cdotp (y\cdotp z)}(\theta_{y:x}(t): \theta_{x\cdotp y}(s)))\cdotp \theta_{(y\cdotp z): (y:x)}(\theta_{y\cdotp z}(t)\cdotp \theta_{z: y}(u)) \nonumber 
\end{eqnarray}}
and, by the hypothesis, it follows that
{\small\begin{eqnarray*}
& &\theta_{(y:x)\cdotp (y\cdotp z)}(\theta_{y:x}(t): \theta_{x\cdotp y}(s)))\cdotp \theta_{(y\cdotp z): (y:x)}(\theta_{y\cdotp z}(t)\cdotp \theta_{z: y}(u)  \\
&=& (\theta_{(y:x)\cdotp (y\cdotp z)}(\theta_{y:x}(t)):\theta_{(y:x)\cdotp (y\cdotp z)}(\theta_{x\cdotp y}(s)))\cdotp (\theta_{(y\cdotp z): (y:x)}(\theta_{y\cdotp z}(t))\cdotp \theta_{(y\cdotp z): (y:x)}(\theta_{z: y}(u)) ) \nonumber \\
&=&(\theta_{(y:x)\cdotp (y\cdotp z)}(\theta_{y:x}(t)):\theta_{(x\cdotp y)\cdotp (x\cdotp z)}(\theta_{x\cdotp y}(s)))\cdotp (\theta_{(y:x)\cdotp (y\cdotp z)}(\theta_{(y:x)}(t))\cdotp \theta_{(z:y): (z:x)}(\theta_{z: y}(u))), \nonumber 
\end{eqnarray*}}
i.e., \eqref{eq:2} holds.
Therefore, since we showed that the first members of \eqref{eq:1} and \eqref{eq:2} coincide
and $X$ is a q-cycle set, we have that equality \eqref{ug1} holds. In a similar way, one can check that \eqref{ug2} and \eqref{ug3} hold, hence it remains to show that $\sigma_{(x,s)}$ is bijective for every $(x,s)\in X\times S$. Suppose that $(x,s)\cdotp(y,t)=(x,s)\cdotp(z,u)$ for some $(x,s),(y,t),(z,u)\in X\times S$. Then,
$$(x\cdotp y,\theta_{x\cdotp y}(s)\cdotp \theta_{y: x}(t))=(x\cdotp z,\theta_{x\cdotp z}(s)\cdotp \theta_{z: x}(u)) $$
and, since $X$ and $S$ are q-cycle sets, it follows that $y=z$ and $\theta_{y:x}(t)=\theta_{z: x}(u)$. Therefore, from the bijectivity of $\theta_{y:x}$, we have that $t=u$, hence $\sigma_{(x,s)}$ is injective. Finally, we have that $(x,s)\cdotp (\sigma_x^{-1}(y),\theta^{-1}_{\sigma_x^{-1}(y):x}(\sigma^{-1}_{\theta_{y}(s)}(t)))=(y,t)$, hence the thesis.
\end{proof}

By semidirect product of q-cycle sets we are able to construct further dynamical extensions.
\begin{ex}
Let $X:=\{0,1,2\}$ be the q-cycle set given by $x\cdotp y:=y$, for all $x,y\in X$, and 
$$
x:0=x:1:=0\qquad\qquad  
x:2=2,
$$
for every $x\in X$. Moreover, let $Y:=\{0,1,2\}$ be the q-cycle set given by 
\begin{align*}
x\cdotp y=x:y=y,
\end{align*}
for all $x,y\in Y$, and $\theta:X\longrightarrow \Aut(Y)$ given by $\theta(0)=\theta(1):=(0\quad 1)$ and $\theta(2):=id_Y$. Then, the semidirect product $(X\times Y,\cdotp,:)$ is a q-cycle set of order $9$. In particular, note that it is not regular.
\end{ex}

Note that, if $X$ and $S$ are cycle sets, the semidirect product (as q-cycle sets) of $X$ and $S$ coincides with the semidirect product of cycle sets developed by Rump in \cite{rump2008semidirect}. 
Moreover, since biquandles in \cite{bardakov2019general} can be viewed as particular q-cycle sets, referring to Problem 4.15 in \cite{bardakov2019general} itself, the previous proposition gives rise to a general definition of semidirect product of biquandles.

\bibliographystyle{elsart-num-sort}
\bibliography{Bibliography}

\end{document}